\documentclass{amsart}

\input xy
\xyoption{all}

\usepackage{amsmath}
\usepackage{amssymb}
\usepackage{url}
\usepackage{amsthm}  
\usepackage[final]{showkeys}

\newtheorem{same}{This should never appear}
\newtheorem{defin}[same]{Definition}

\newtheorem{remark}[same]{Remark}
\newtheorem{theorem}[same]{Theorem}
\newtheorem{corollary}[same]{Corollary}

\newtheorem{lemma}[same]{Lemma}
\newtheorem{fact}[same]{Fact}

\newtheorem{prop}[same]{Proposition}

\newtheorem{hyp}[same]{Hypothesis}

\newtheorem{defin*}{Definition}
\newtheorem*{theorem*}{Theorem}

\newbox\noforkbox \newdimen\forklinewidth
\setbox0\hbox{$\textstyle\smile$}
\setbox1\hbox to \wd0{\hfil\vrule width \forklinewidth depth-2pt
 height 10pt \hfil}
\setbox\noforkbox\hbox{\lower 2pt\box1\lower 2pt\box0\relax}
\def\unionstick{\mathop{\copy\noforkbox}\limits}

\def\nonfork_#1{\unionstick_{\textstyle #1}}

\setbox0\hbox{$\textstyle\smile$}
\setbox1\hbox to \wd0{\hfil{\sl /\/}\hfil}
\setbox2\hbox to \wd0{\hfil\vrule height 10pt depth -2pt width
              \forklinewidth\hfil}
\newbox\doesforkbox
\setbox\doesforkbox\hbox{\lower 2pt\box1 \lower 2pt\box2\lower2pt\box0\relax}
\def\nunionstick{\mathop{\copy\doesforkbox}\limits}

\def\fork_#1{\nunionstick_{\textstyle #1}}

\newcommand{\Mod}{\textrm{Mod}}

\newcommand{\cf}{\operatorname{cf }}
\newcommand{\rest}{\upharpoonright}

\newcommand{\otp}{\operatorname{otp}}

\newcommand{\Aut}{\operatorname{Aut}}
\newcommand{\C}{\operatorname{\mathfrak{C}}}

\def\b{\mathbf{b}}

\newcommand{\K}{\mathcal{K}}
\newcommand{\leap}[1]{\prec_{#1}}
\newcommand{\ltap}[1]{\precneqq_{#1}}

\newcommand{\lea}{\leap{\K}}
\newcommand{\leas}{\leap{\K^\ast}}
\newcommand{\ltas}{\ltap{\K^\ast}}

\newcommand{\LS}{\operatorname{LS}}
\newcommand{\EM}{\operatorname{EM}}
\newcommand{\gS}{\operatorname{gS}}

\newcommand{\dnf}{\unionstick}

\newcommand{\adnf}{\dnf^*}

\newcommand{\Ll}{\mathbb{L}}

\newcommand{\seq}[1]{\langle #1 \rangle}

\newcommand{\gtp}{\operatorname{gtp}}

\newcommand{\fct}[2]{{}^{#1}#2}

\title[Superstability from categoricity in AECs]{Superstability from categoricity in abstract elementary classes}

\author{Will Boney}
\address{Department of Mathematics \\ Harvard University \\ Cambridge, Massachusetts, USA}
\email{wboney@math.harvard.edu}
\thanks{This material is based upon work done while the first author was supported by the National Science Foundation under Grant No. DMS-1402191.}

\author{Rami Grossberg}
\address{Department of Mathematical Sciences \\ Carnegie Mellon University \\ Pittsburgh, Pennsylvania, USA}
\email{rami@cmu.edu}

\author{Monica M. VanDieren}
\address{Department of Mathematics\\
Robert Morris University\\
Moon Township Pennsylvania USA}
\email{vandieren@rmu.edu}

\author{Sebastien Vasey}
\address{Department of Mathematical Sciences \\ Carnegie Mellon University \\ Pittsburgh, Pennsylvania, USA}
\email{sebv@cmu.edu}

\date{\today}


\begin{document}


\begin{abstract}

  Starting from an abstract elementary class with no maximal models, Shelah and Villaveces have shown (assuming instances of diamond) that categoricity implies a superstability-like property for nonsplitting, a particular notion of independence. We generalize their result as follows: given any abstract notion of independence for Galois (orbital) types over models, we derive that the notion satisfies a superstability property provided that the class is categorical and satisfies a weakening of amalgamation. This extends the Shelah-Villaveces result (the independence notion there was splitting) as well as a result of the first and second author where the independence notion was coheir. The argument is in ZFC and fills a gap in the Shelah-Villaveces proof.
\end{abstract}

\maketitle

%

\section{Introduction}

\subsection{General motivation and history}

Forking is one of the central notions of model theory, discovered and developed by Shelah in the seventies for stable and NIP theories \cite{shelahfobook78}. 
One way to extend Shelah's first-order stability theory is to move beyond first-order.  In the mid seventies, Shelah did this by starting the program of \emph{classification theory for non-elementary classes} focusing first on classes axiomatizable in $\Ll_{\omega_1, \omega}(\mathbf{Q})$ \cite{sh48} and later on the more general abstract elementary classes (AECs) \cite{sh88}. Roughly, an AEC is a pair $\K = (K, \lea)$ satisfying some of the basic category-theoretic properties of $(\Mod(T), \prec)$ (but not the compactness theorem). Among  the central problems, there are the decades-old categoricity and eventual categoricity conjectures of Shelah.  In this paper, we assume that the reader has a basic knowledge of AECs, see for example \cite{grossberg2002} or \cite{baldwinbook09}.

One key shift in this program is the move away from syntactic types (studied in the $\Ll_{\lambda^+, \omega}$ context by \cite{sh16, grsh222, grsh238} and others) and towards a semantic notion of type, introduced in \cite{sh300-orig} and named \emph{Galois type} by Grossberg \cite{grossberg2002}.\footnote{Shelah uses the name \emph{orbital types} in some later papers.} This has an easy definition when the class $\K$ has amalgamation, joint embedding and no maximal models, as these properties allow us to assume that all the elements of $\K$ we would like to discuss are substructures of a ``monster'' model $\C\in \K$. In that case, $\gtp(\b/A)$ is defined as the orbit of $\b$ under the action of the group $\Aut _A (\C)$ on $\C$. One can also develop the notion of Galois type without the above assumption, however then the definition is more technical.



\subsection{Independence, superstability, and no long splitting chains in AECs}\label{indep-nolong}

In \cite{sh394} a first candidate for an independence relation was introduced: the notion of $\mu$-splitting (for $M_0 \lea M$ both in $\K_\mu$, $p\in\gS(M)$ \emph{$\mu$-splits over $M_0$} provided there are $M_0\lea M_\ell \lea M$, $\ell=1,2$ and $f: M_1\cong_{M_0}M_2$ such that $f(p\restriction M_1)\neq p\restriction M_2$).

This notion was used by Shelah to establish a downward version of his categoricity conjecture from a successor for classes having the amalgamation property. Later similar arguments \cite{tamenesstwo, tamenessthree} were used to derive a strong upward version of Shelah's conjecture for classes satisfying the additional locality property of (Galois) types called tameness.

In Chapter II of \cite{shelahaecbook}, Shelah introduced good $\lambda$-frames: an axiomatic definition of forking on Galois types over models of size $\lambda$. The notion is, by definition, required to satisfy basic properties of forking in superstable first-order theories (e.g.\ symmetry, extension, uniqueness, and local character). The theory of good $\lambda$-frames is well-developed and has had several applications to the categoricity conjecture (see Chapters III and IV of \cite{shelahaecbook} and recent work of the fourth author \cite{ap-universal-v10, categ-universal-2-v3-toappear, categ-primes-v4, downward-categ-tame-apal}).

Constructions of good frames rely on weaker independence notions like nonsplitting, see e.g.\ \cite{ss-tame-jsl, vv-symmetry-transfer-v3}.  A key property of splitting in these constructions is that there is ``no long splitting chains in $\K_\mu$'': if $\seq{M_i : i \le \alpha}$ is an increasing continuous chain in $\K_\mu$ (so $\alpha < \mu^+$ is a limit ordinal) and $M_{i + 1}$ is universal over $M_i$ for each $i < \alpha$, then for any $p \in \gS (M_\alpha)$ there exists $i < \alpha$ so that $p$ does not $\mu$-split over $M_i$ (this is called \emph{strong universal local character at $\alpha$} in the present paper, see Definition \ref{indep-def}). This can be seen as a replacement for the statement ``every type does not fork over a finite set''. The property is already studied in \cite{sh394}, and has several nontrivial consequences: for example (assuming amalgamation, joint embedding, no maximal models, stability in $\mu$, and tameness), no long splitting chains in $\K_\mu$ implies that $\K$ is stable everywhere above $\mu$ \cite[Theorem 5.6]{ss-tame-jsl} and  has a good $\mu^+$-frame on the subclass of saturated models of cardinality $\mu^+$ \cite[Corollary 6.14]{vv-symmetry-transfer-v3}. No long splitting chains has consequences for the uniqueness of limit models, another superstability-like property saying in essence that saturated models can be built in few steps (see for example \cite{shvi635, vandierennomax, nomaxerrata, vandieren-symmetry-apal}).

The first and second authors have explored another approach to independence by adapting the notion of coheir to AECs.  They have shown that for classes satisfying amalgamation which are also tame and short (a strengthening of tameness, using the variables of a type instead of its parameters), failure of a certain order property implies that coheir has some basic properties of forking from a stable first-order theory. There the ``no long coheir chain'' property also has strong consequences (for example on the uniqueness of limit models \cite[Corollary 6.18]{bg-v11}).

\subsection{No long splitting chains from categoricity}

It is natural to ask whether no long splitting chains (or no long coheir chains) in $\K_\mu$ follows from categoricity above $\mu$. Shelah has shown that this holds for splitting (assuming amalgamation and no maximal models) if the categoricity cardinal has cofinality greater than $\mu$  \cite[Lemma 6.3]{sh394}. Without any cofinality restriction, a breakthrough was made in a paper of Shelah and Villaveces when they proved no long splitting chains assuming no maximal models and instances of diamond  \cite[Theorem 2.2.1]{shvi635}. Later, Boney and Grossberg  used the Shelah-Villaveces argument to derive the result in their context also for coheir \cite[Theorem 6.8]{bg-v11}. It was also observed  that the Shelah-Villaveces argument does not need diamond if one assumes full amalgamation \cite[5.3]{gv-superstability-v4}. In conclusion we have:

\begin{fact}\label{known-results}
  Let $\K$ be an AEC with no maximal models. Let $\LS (\K) \le \mu < \lambda$ and assume that $\K$ is categorical in $\lambda$.

  \begin{enumerate}
  \item\label{shvi-known} \cite[Theorem 2.2.1]{shvi635} If $\Diamond_{S_{\cf{\mu}}^{\mu^+}}$ holds then $\K$ has no long splitting chains in $\K_\mu$.
  \item\label{bg-known} \cite[Theorem 6.8]{bg-v11} If $\K$ has amalgamation, $\kappa \in (\LS (\K), \mu)$, $\K$ does not have the weak $\kappa$-order property and is fully $(<\kappa)$-tame and short, then $\K$ has no long coheir chains in $\K_\mu$.
  \item\label{gv-known} \cite[Corollary 5.3]{gv-superstability-v4} If $\K$ has amalgamation, then $\K$ has no long splitting chains in $\K_\mu$.
  \end{enumerate}
\end{fact}
\begin{remark}
Fact \ref{known-results} has applications to more ``concrete'' frameworks than AECs. One can deduce from it (and the aforementioned fact that no long splitting chains implies stability on a tail in the presence of tameness) an alternate proof that a first-order theory $T$ categorical above $|T|$ is superstable. More generally, one obtains the same statement for the class $K$ of models of a homogeneous diagram in $T$ \cite{sh3}. The later was open for $|T|$ uncountable and $K$ categorical in $\aleph_{\omega} (|T|)$ (see \cite[Section 4]{categ-primes-v4}).
\end{remark}

\subsection{Gaps in the Shelah-Villaveces proof}\label{gap-section}

In a preliminary version of \cite{bg-v11}, the proof of Theorem 6.8 referred to the argument used in \cite[Theorem 2.2.1]{shvi635}.  The referee of  \cite{bg-v11} insisted that the full argument necessary for Theorem 6.8 be included. After looking closely at the argument in \cite{shvi635}, we concluded that there was a small gap in the division of cases and a need to specify the exact use of the club guessing principle that they imply.

More specifically, Shelah and Villaveces \cite[Theorem 2.2.1]{shvi635} assume for a contradiction that no long splitting chains fails and can divide the situation into three cases, (a), (b), and (c). In the division into cases \cite[Claim 2.2.3]{shvi635}, just after the statement of property $\otimes_{i}$, Shelah and Villaveces claim that they can ``repeat the procedure above'' on a certain chain of models of length $\mu$. However the ``procedure above'' was used on a chain of length $\sigma$, where $\sigma$ is a \emph{regular} cardinal and regularity was used in the proof. As $\mu$ is a potentially singular cardinal, there is a problem.

Once the division of cases is done, Shelah and Villaveces prove that cases (a), (b), (c) contradict categoricity. When proving this for (b), they use a club-guessing principle for $\mu^+$ on the stationary set of points of cofinality $\sigma$ (see Fact \ref{club-guess-fact-1}). The principle only holds when $\sigma < \mu$, so the case $\sigma = \mu$ is missing.

\subsection{Statement and discussion of the main theorem}

In this paper, we give a generalized, detailed, and corrected proof of Fact \ref{known-results} that does not rely on any of the material in \cite{shvi635}. The key definitions are given at the start of the next section and the first seven hypotheses are collected in Hypothesis \ref{hyp}.

\begin{theorem}[Main Theorem]\label{shvi-technical}
  \emph{If}:

  \begin{enumerate}
  \item\label{hyp-1} $\K$ is an AEC.
  \item\label{hyp-2} $\mu \ge \LS (\K)$.
  \item\label{hyp-3} For every $M \in \K_{\mu}$, there exists an amalgamation base $M' \in \K_{\mu}$ such that $M \lea M'$.
  \item\label{hyp-4} For every amalgamation base $M \in \K_{\mu}$, there exists an amalgamation base $M' \in \K_{\mu}$ such that $M'$ is universal over $M$.
  \item\label{hyp-5} Every limit model in $\K_{\mu}$ is an amalgamation base.
  \item\label{hyp-6} $\adnf$ is as in Definition \ref{indep-def} with $\K^\ast$ the class of amalgamation bases in $\K_\mu$ (ordered with the strong substructure relation inherited from $\K$).
  \item\label{hyp-7} $\adnf$ satisfies invariance $(I)$ and monotonicity $(M)$.
  \item\label{hyp-8} $\adnf$ has weak universal local character at \emph{some} cardinal $\sigma < \mu^+$.
  \item\label{hyp-9} $\K$ has an Ehrenfeucht-Mostowski (EM) blueprint $\Phi$ with $|\tau (\Phi)| \le \mu$ such that every $M \in \K_{[\mu, \mu^+]}$ embeds inside $\EM_{\tau} (\mu^+, \Phi)$ (where we write $\tau := \tau (\K)$).
  \end{enumerate}

  \emph{Then} $\adnf$ has \emph{strong} universal local character at \emph{all} limit ordinals $\alpha < \mu^+$.
\end{theorem}
\begin{remark}
  As in \cite{shvi635}, when we say that $M$ is an \emph{amalgamation base} we mean that it is an amalgamation base in the class $\K_{\|M\|}$, i.e.\ we do \emph{not} require that larger models can be amalgamated over $M$.
\end{remark}

Some of the hypotheses of Theorem \ref{shvi-technical} may appear technical. Let us give a little more motivation. 
\begin{itemize}
	\item Hypotheses (\ref{hyp-3}-\ref{hyp-5}) are the statements that Shelah and Villaveces derive (assuming instances of diamond) from categoricity and no maximal models. It is well known that they hold in AECs with amalgamation. 
	\item Hypothesis (\ref{hyp-4}) implies stability in $\mu$. 
	\item Hypothesis (\ref{hyp-8}) can be seen as a consequence of stability (akin to ``every type does not fork over a set of size at most $\mu$''). 
	\item Hypothesis (\ref{hyp-9}) follows from categoricity (see the proof of Corollary \ref{shvi-contexts}).  In fact, it is strictly weaker: for a first-order theory $T$, (\ref{hyp-9}) holds if and only if $T$ is superstable by \cite[Section 5]{gv-superstability-v4}.
\end{itemize}

How are the gaps mentioned in Section \ref{gap-section} addressed in our proof of Theorem \ref{shvi-technical}? The first gap (in the division into cases) is fixed in Lemma \ref{forking-calc-lem-part1}.(\ref{part13}). The second gap (in the use of the club guessing principle) is addressed here by a division into cases in the proof of Theorem \ref{shvi-technical} at the end of this paper: there we use Lemma \ref{technical-solv-lem} only when $\alpha < \sigma$.

Before starting to prove Theorem \ref{shvi-technical}, we give several contexts in which its hypotheses hold. This shows in particular that Fact \ref{known-results} follows from Theorem \ref{shvi-technical}.

\begin{corollary}\label{shvi-contexts}
  Let $\K$ be an AEC with arbitrarily large models. Let $\LS (\K) \le \mu < \lambda$ and assume that $\K$ is categorical in $\lambda$ and $\K_{<\lambda}$ has no maximal models. Then:

  \begin{enumerate}
  \item If $\diamondsuit_{S_{\cf{\mu}}^{\mu^+}}$ holds, then the hypotheses of Theorem \ref{shvi-technical} hold with $\adnf$ being non-$\mu$-splitting.
  \item If $\K_{\mu}$ has amalgamation, then:
    \begin{enumerate}
      \item The hypotheses of Theorem \ref{shvi-technical} hold with $\adnf$ being non-$\mu$-splitting.
      \item If $\kappa \in (\LS (\K), \mu)$ is such that $\K$ does not have the weak $\kappa$-order property, then the hypotheses of Theorem \ref{shvi-technical} hold with $\adnf$ being $(<\kappa)$-coheir (see \cite{bg-v11}).
    \end{enumerate}
    \end{enumerate}
\end{corollary}
\begin{proof}
  Fix an EM blueprint $\Psi$ for $\K$ (with $|\tau (\Psi)| \le \mu$). We first show that there exists an EM blueprint $\Phi$ with $|\tau (\Phi)| \le \mu$ such that any $M \in \K_{[\mu, \mu^+]}$ embeds inside $\EM_{\tau} (\mu^+, \Phi)$. Let $M \in \K_{[\mu, \mu^+]}$. Using no maximal models and categoricity,  $M$ embeds inside $\EM_{\tau} (\lambda, \Psi)$, and hence inside $\EM_{\tau} (S, \Psi)$ for some $S \subseteq \lambda$ with $|S| \le \mu^+$. Therefore $M$ also embeds inside $\EM_{\tau} (\alpha, \Psi)$, where $\alpha := \otp (S) < \mu^{++}$. Now it is well known (see e.g.\ \cite[Claim 15.5]{baldwinbook09}) that $\alpha$ embeds inside $\EM_{\tau} (\fct{<\omega}{\mu^{+}}, \Phi)$. The class $\{\fct{<\omega}{I} \mid I \text { is a linear order}\}$ is an AEC, therefore by composing EM blueprints there exists an EM blueprint $\Phi$ for $\K$ such that $|\tau (\Phi)| \le \mu$ and $\EM_{\tau} (I, \Phi) = \EM_{\tau} (\fct{<\omega}{I}, \Psi)$ for any linear order $I$. In particular, $M$ embeds inside $\EM_{\tau} (\mu^{+}, \Phi)$, as desired.

  As for the hypotheses on density of amalgamation bases, existence of universal extension, and limit models being amalgamation bases, in the first context this is proven in \cite{shvi635} (note that $\diamondsuit_{S_{\cf{\mu}}^{\mu^+}}$ implies $2^{\mu} = \mu^+$). When $\K_{\mu}$ has full amalgamation, existence of universal extension is due to Shelah. It is stated (but not proven) in \cite[Lemma 2.2]{sh394}; see \cite[Lemma 10.5]{baldwinbook09} for a proof.

  In all the contexts given, it is trivial that $\adnf$ satisfies $(I)$ and $(M)$. In the first context, it can be shown that non $\mu$-splitting has weak universal local character at any $\sigma < \mu^+$ such that $2^{\sigma} > \mu$ (see the proof of case (c) in \cite[Theorem 2.2.1]{shvi635} or \cite[Lemma 12.2]{baldwinbook09}). Of course, this also holds when $\K_{\mu}$ has full amalgamation. As for $(<\kappa)$-coheir, it has weak universal local character at any $\sigma < \mu^+$ such that $2^\sigma > \kappa$. This is given by the proof of \cite[Theorem 6.8]{bg-v11} (note that using a back and forth argument, one can assume without loss of generality that any $M_{i + 1}$ in the chain is $\kappa$-saturated).
\end{proof}

\subsection{Other advantages of the main theorem}

As should be clear from Corollary \ref{shvi-contexts}, another advantage of the main theorem is that it separates the combinatorial set theory from the model theory (it holds in ZFC) and also shows that there is nothing special about splitting in \cite{shvi635}.

Some results here are of independent interest. For example, any independence relation satisfying invariance and monotonicity has (assuming categoricity) a certain continuity property (see Lemma \ref{technical-solv-lem}).

\subsection{Acknowledgments}

We thank the referee for comments that helped improve the presentation of this work.

This paper was written while the fourth author was working on a Ph.D.\ thesis under the direction of the second author at Carnegie Mellon University and he would like to thank Professor Grossberg for his guidance and assistance in his research in general and in this work specifically.

\section{Proof of the main theorem}

We now define the weak framework for independence that we use.

\begin{defin}\label{indep-def}
Let $\K^\ast$ be an abstract class\footnote{That is, a partial order $(K^\ast, \leas)$ such that $K^\ast$ is a class of structures in a fixed vocabulary closed under isomorphisms, $\leas$ is invariant under isomorphisms, and $M \leas N$ implies that $M$ is a substructure of $N$.} and $\adnf$ be a 4-ary relation such that if $a \adnf_{M_0}{}^N M$ holds, then $M_0 \leas M \leas N$ are all in $\K^\ast$ and $a \in |N|$.
\begin{enumerate}
\item The following are several properties we will assume about $\adnf$ (but we will always mention when we assume them).
  \begin{enumerate}
    	\item $\adnf$ has \emph{invariance} $(I)$ if it is preserved under isomorphisms: if $a \adnf_{M_0}{}^N M$ and $f: N \cong N'$, then $f (a) \adnf_{f[M_0]}{}^{N'} f[M]$.
	\item $\adnf$ has \emph{monotonicity} $(M)$ if: 
	\begin{enumerate}
		\item If $a \adnf_{M_0}{}^N M$, $M_0 \leas M_0' \leas M' \leas M$, and $N \leas N'$, then $a \dnf_{M'_0} {}^{N'} M'$; and:
		\item If $a \adnf_{M_0}{}^N M$, $N' \leas N$ is such that $M \leas N'$ and $a \in |N'|$, then $a \adnf_{M_0}{}^{N'} M$.
	\end{enumerate}
  \end{enumerate}
\item $(I)$ and $(M)$ mean that this relation is really about Galois types, so we write \emph{$\gtp (a/M; N)$ does not $*$-fork over $M_0$} for $a \adnf_{M_0}{}^N M$.
\item\label{weak-local} For a limit ordinal $\alpha$, $\adnf$ has \emph{weak universal local character at $\alpha$} if for any increasing continuous sequence $\seq{M_ i \in \K^\ast \mid i \leq \alpha}$ and any type $p \in \gS (M_\alpha)$, if $M_{i + 1}$ is universal over $M_i$ for each $i < \alpha$, then there is some $i_0 < \alpha$ such that $p \rest M_{i_0+1}$ does not $*$-fork over $M_{i_0}$.
\item\label{strong-local} For a limit ordinal $\alpha$, $\adnf$ has \emph{strong universal local character at $\alpha$} if for any increasing continuous sequence $\seq{M_ i \in \K^\ast \mid i \leq \alpha}$ and any type $p \in \gS (M_\alpha)$, if $M_{i + 1}$ is universal over $M_i$ for each $i < \alpha$, then there is some $i_0 < \alpha$ such that $p$ does not $*$-fork over $M_{i_0}$.
\end{enumerate}
\end{defin}
\begin{remark}\label{loc-rmk} \
  \begin{enumerate}
  \item In the setup of Fact \ref{known-results}.(\ref{shvi-known}), non-$\mu$-splitting on the class $\K^\ast$ of amalgamation bases of cardinality $\mu$ will have $(I)$ and $(M)$, see Fact \ref{shvi-contexts}.
  \item If $\alpha < \beta$ are limit ordinals and $\adnf$ has weak universal local character at $\alpha$, then $\adnf$ has weak universal local character at $\beta$, but this need not hold for strong universal local character (if say $\cf{\beta} < \cf{\alpha}$).
  \item If $\adnf$ has $(M)$ and $\adnf$ has strong universal local character at $\cf{\alpha}$, then $\adnf$ has strong universal local character at $\alpha$.
  \item If $\adnf$ has $(M)$, strong universal local character at $\alpha$ implies weak universal local character at $\alpha$.
  \item If (as will be the case in this note) $\K^\ast$ is a class of structures of a fixed size $\mu$, then we only care about the properties when $\alpha < \mu^+$.
  \end{enumerate}
\end{remark}

We collect the first seven hypotheses of Theorem \ref{shvi-technical} into a hypothesis that will be assumed for the rest of the paper.

\begin{hyp}\label{hyp}\
  \begin{enumerate}
\item\label{hyp-hyp-1} $\K$ is an AEC.
  \item\label{hyp-hyp-2} $\mu \ge \LS (\K)$.
  \item\label{hyp-hyp-3} For every $M \in \K_{\mu}$, there exists an amalgamation base $M' \in \K_{\mu}$ such that $M \lea M'$.
  \item\label{hyp-hyp-4} For every amalgamation base $M \in \K_{\mu}$, there exists an amalgamation base $M' \in \K_{\mu}$ such that $M'$ is universal over $M$.
  \item\label{hyp-hyp-5} Every limit model in $\K_{\mu}$ is an amalgamation base.
  \item\label{hyp-hyp-6} $\adnf$ is as in Definition \ref{indep-def} with $\K^\ast$ the class of amalgamation bases in $\K_\mu$ (ordered with the strong substructure relation inherited from $\K$).
  \item\label{hyp-hyp-7} $\adnf$ satisfies invariance $(I)$ and monotonicity $(M)$.

\end{enumerate}
\end{hyp}

The proof of Theorem \ref{shvi-technical} can be decomposed into two steps. First, we study two more variations on local character: continuity and absence of alternations. We show that if strong local character fails but enough weak local character holds, then there must be some failure of continuity, or some alternations. Second, we show that categoricity (or more precisely the existence of a universal EM model in $\mu^+$) implies continuity and absence of alternations. The first step uses the weak local character (but not categoricity, it is essentially forking calculus) but the second does not (but does use categoricity).

The precise definitions of continuity and alternations are as follows.

\begin{defin}
  Let $\K^\ast$ and $\adnf$ be as in Definition \ref{indep-def} and let $\alpha$ be a limit ordinal.

  \begin{enumerate}
  \item $\adnf$ has \emph{universal continuity at $\alpha$} if for any increasing continuous sequence $\seq{M_ i \in \K^\ast \mid i \leq \alpha}$ and any type $p \in \gS (M_\alpha)$, if for each $i < \alpha$ $M_{i + 1}$ is universal over $M_i$ and $p \rest M_i$ does not $\ast$-fork over $M_0$, then $p$ does not $\ast$-fork over $M_0$.
  \item For $\delta < \mu^+$ a limit, $\adnf$ has \emph{no $\delta$-limit alternations at $\alpha$} if for any increasing continuous sequence $\seq{M_ i \in \K^\ast \mid i \leq \alpha}$ with $M_{i + 1}$ $(\mu, \delta)$-limit over $M_i$ for all $i < \alpha$ and any type $p \in \gS (M_\alpha)$, there exists $i < \alpha$ such that the following fails: $p \rest M_{2i + 1}$ $\ast$-forks over $M_{2i}$ and $p \rest M_{2i + 2}$ does not $\ast$-fork over $M_{2i + 1}$. If this fails, we say that \emph{$\adnf$ has $\delta$-limit alternations at $\alpha$}.
  \end{enumerate}
\end{defin}

Note that the failure of universal continuity and no $\delta$-limit alternation correspond respectively to cases (a) and (b) in the proof of \cite[Theorem 2.2.1]{shvi635}. Case (c) there corresponds to failure of weak universal local character at $\mu$ (which is assumed to hold here, see (\ref{hyp-8}) of Theorem \ref{shvi-technical}).

The following technical lemmas and proposition implement the first step described after the statement of Hypothesis \ref{hyp}.  In particular, Proposition \ref{forking-calc-lem-part3} below says that if we can prove weak local character at some $\sigma$, continuity and no alternations at \emph{all} $\alpha$, then strong local character at all $\alpha$ follows. Lemma \ref{forking-calc-lem-part1} is a collection of preliminary steps toward proving Proposition \ref{forking-calc-lem-part3}. Lemma \ref{forking-calc-lem-part0} is used separately in the proof of the main theorem (it says that weak universal local character implies the absence of alternations).  Throughout, recall that we are assuming Hypothesis \ref{hyp}.

\begin{lemma}\label{forking-calc-lem-part0}
Let 
$\sigma < \mu^+$ be a (not necessarily regular) cardinal and $\delta < \mu^+$ be a limit ordinal.  If $\adnf$ has weak universal local character at $\sigma$, then $\adnf$ has no $\delta$-limit alternations at $\sigma$.
\end{lemma}

\begin{proof}
Fix $\seq{M_i : i \le \alpha}$, $\delta$, $p$ as in the definition of having no $\delta$-limit alternations. Apply weak universal local character to the chain $\seq{M_{2i} : i \le \alpha}$.
\end{proof}

We now outline the proof of Proposition \ref{forking-calc-lem-part3}. Again, it may be helpful to remember that we will later prove that (in the context of Theorem \ref{shvi-technical}) continuity holds at all lengths and that there are no alternations.

Two important basic results are 
\begin{itemize}
	\item continuity together with weak local character imply strong local character at \emph{regular length} (Lemma \ref{forking-calc-lem-part1}.(\ref{part1})); and 
	\item it does not matter whether in the definition of weak and strong universal local character we require ``$M_{i + 1}$ limit over $M_i$'' or ``$M_{i + 1}$ universal over $M_i$,'' and the length of the limit models does not matter (Lemma \ref{forking-calc-lem-part1}.(\ref{part11})).
\end{itemize}
The first of these is proven by contradiction, and the second is a straightforward argument using universality. 

Assume for a moment we have strong universal local character at some limit length $\gamma$. Let us try to prove weak universal local character at (say) $\omega$ (then we can use the first basic result to get the strong version, assuming continuity). By the second basic result, we can assume we are given an increasing continuous sequence $\seq{M_n : n \le \omega}$ with $M_{n + 1}$ $(\mu, \gamma)$-limit over $M_n$ for all $n < \omega$ and $p \in \gS (M_\omega)$. By the strong universal local character assumption we know that $p \rest M_{n  +1}$ does not $\ast$-fork over some intermediate model between $M_n$ and $M_{n + 1}$, so if we assume that $p \rest M_{n + 1}$ $\ast$-forks over $M_n$ for all $n < \omega$, we will end up getting alternations. This is the essence of Lemma \ref{forking-calc-lem-part1}.(\ref{part2}).

Thus to prove strong universal local character at \emph{all} cardinals, it is enough to obtain it at \emph{some} cardinal. Fortunately in the hypothesis of Proposition \ref{forking-calc-lem-part3}, we are already assuming weak universal local character at some $\sigma$. If $\sigma$ is regular we are done by the first basic result, but unfortunately $\sigma$ could be singular. In this case Lemma \ref{forking-calc-lem-part1}.(\ref{part13}) (using Lemma \ref{forking-calc-lem-part1}.(\ref{part12}) as an auxiliary claim) shows that failure of strong universal local character at $\sigma$ implies alternations, even when $\sigma$ is singular.

\begin{lemma}\label{forking-calc-lem-part1}
Let $\alpha < \mu^+$ be a regular cardinal, $\sigma < \mu^+$ be a (not necessarily regular) cardinal, and $\delta < \mu^+$ be a limit ordinal.
  \begin{enumerate}
  	  \item\label{part1} If $\adnf$ has universal continuity at $\alpha$ and weak universal local character at $\alpha$, then $\adnf$ has strong universal local character at $\alpha$.
  \item\label{part11} We obtain an equivalent definition of weak [strong] universal local character at $\sigma$, if in Definition \ref{indep-def}.(\ref{weak-local}) [\ref{indep-def}.(\ref{strong-local})] we ask in addition that ``$M_{i + 1}$ is $(\mu, \delta)$-limit over $M_i$'' for all $i < \sigma$. 
  \item\label{part12} Assume that $\adnf$ has weak universal local character at $\sigma$. Let $\seq{M_i : i \le \sigma}$ be increasing continuous in $\K^\ast$ with $M_{i + 1}$ universal over $M_i$ for all $i < \sigma$. For any $p \in \gS (M_\sigma)$ there exists a \emph{successor} $i < \sigma$ such that $p \rest M_{i + 1}$ does not $\ast$-fork over $M_i$.
  \item\label{part13} If $\adnf$ has universal continuity at $\sigma$, weak universal local character at $\sigma$, and no $\delta$-limit alternations at $\omega$, then $\adnf$ has strong universal local character at $\sigma$.
  \item\label{part2} Assume that $\adnf$ has strong universal local character at $\sigma$. If $\adnf$ does \emph{not} have weak universal local character at $\alpha$, then $\adnf$ has $\sigma$-limit alternations at $\alpha$.
  \end{enumerate}
\end{lemma}

\begin{proof}\
\begin{enumerate}
  \item Suppose that $\seq{M_i : i \le \alpha}$, $p$ is a counterexample. \\

    {\bf Claim:} For each $i < \alpha$, there exists $j_i \in (i, \alpha)$ such that $p \rest M_{j_i}$ $\ast$-forks over $M_i$.  \\

    {\bf Proof of Claim:} If $i < \alpha$ is such that for all $j \in (i, \alpha)$, $p \rest M_{j}$ does not $\ast$-fork over $M_i$, then applying universal continuity at $\alpha$ on the chain $\seq{M_k : k \in [i, \alpha]}$ we would get that $p$ does not $\ast$-fork over $M_i$, contradicting the choice of $\seq{M_i : i \le \alpha}$, $p$. \hfill $\dag_{\text{Claim}}$\\

    Now define inductively for $i \le \alpha$, $k_0 := 0$, $k_{i + 1} := j_{k_i}$, and when $i$ is limit $k_i := \sup_{j < i} k_j$. Note that $\seq{k_i : i \le \alpha}$ is strictly increasing continuous and $i < \alpha$ implies $k_i < \alpha$ (this uses regularity of $\alpha$; when $\alpha$ is singular, see (\ref{part13})).

    Apply weak universal local character to the chain $\seq{M_{k_i} : i \le \alpha}$ and the type $p$. We get that there exists $i < \alpha$ such that $p \rest M_{k_{i + 1}}$ does not $\ast$-fork over $M_{k_i}$. This is a contradiction since $k_{i + 1} = j_{k_i}$ and we chose $j_{k_i}$ so that $p \rest M_{j_{k_i}}$ $\ast$-forks over $M_{k_i}$.
  \item We prove the result for weak universal local character, and the proof for the strong version is similar. Fix $\seq{M_i^0 : i \le \sigma}$, $p$ witnessing failure of weak universal local character at $\sigma$. We build a witness of failure $\seq{M_i : i \le \sigma}$, $p$ such that $M_\sigma = M_\sigma^0$, and $M_{i + 1}$ is $(\mu, \delta)$-limit over $M_i$ for each $i < \alpha$. Using existence of universal extensions, we can extend each $M^0_i$ to $M^*_i$ that is $(\mu, \delta)$-limit over $M^0_i$.  Since $M^0_{i+1}$ is universal over $M^0_i$, we can find $f_i:M^*_{i+1} \to_{M^0_i} M^0_{i+1}$.  Since limit models are amalgamation bases, $f_i(M^*_{i+1})$ is an amalgamation base. Now set $M_i^1 := M_i^0$ for $i \leq \sigma$ limit or $0$ and $M_{i+1}^1 := f_i(M^*_{i+1})$.  This is an increasing continuous chain of amalgamation bases with $M_{i+1}^1$ $(\mu, \delta)$-limit over $M_i^1$. Let $M_i := M_{2i}^1$.

    This works: if there was an $i < \sigma$ such that $p \rest M_{i + 1}$ does not $\ast$-fork over $M_i$, this would mean that $p \rest M_{2i + 2}^1$ does not $\ast$-fork over $M_{2i}^1$, but since $M_{2i}^1 \leas M_{2i + 1}^0 \leas M_{2i + 2}^0 \leas M_{2i + 2}^1$, we have by $(M)$ that $p \rest M_{2i + 2}^0$ does not $\ast$-fork over $M_{2i + 1}^0$, a contradiction.
  \item Apply weak universal local character to the chain $\seq{M_{2i} : i < \sigma}$ to get $j < \sigma$ such that $p \rest M_{2j + 2}$ does not $\ast$-fork over $M_{2j}$. By monotonicity, this implies that $p \rest M_{2j + 2}$ does not $\ast$-fork over $M_{2j + 1}$. Let $i := 2j + 1$.
  \item Suppose not, and let $\seq{M_i : i \le \sigma}, p$ be a counterexample. By (\ref{part11}), without loss of generality $M_{i + 1}$ is $(\mu, \delta)$-limit over $M_i$ for all $i < \delta$. As in the proof of (\ref{part1}), for each $i < \sigma$, there exists $j_i \in [i, \sigma)$ such that $p \rest M_{j_i}$ $\ast$-forks over $M_i$. On the other hand, applying (\ref{part12}) to the chain $\seq{M_{j} : j \in [j_i, \sigma]}$, for each $i < \sigma$, there exists a \emph{successor} ordinal $k_i \ge j_i$ such that $p \rest M_{k_i + 1}$ does not $\ast$-fork over $M_{k_i}$. Define by induction on $n \le \omega$, $m_0 := 0$, $m_{2n + 1} := k_{m_{2n}}$, $m_{2n + 2} := k_{m_{2n}} + 1$, and $m_\omega := \sup_{n < \omega} m_n$. By construction, the sequence $\seq{M_{m_n} : n \le \omega}$ witnesses that $\adnf$ has $\delta$-limit alternations at $\omega$, a contradiction.
  \item Let $\gamma := \sigma \cdot \sigma$. By (\ref{part11}), there exists $\seq{M_i : i \le \alpha}$, $p$ witnessing failure of weak universal local character at $\alpha$ such that for all $i < \alpha$, $M_{i + 1}$ is $(\mu, \gamma)$-limit over $M_i$. Let $\seq{M_{i, j} : j \le \gamma}$ witness that $M_{i + 1}$ is $(\mu, \gamma)$-limit over $M_i$ (i.e.\ it is increasing continuous with $M_{i, j +1}$ universal over $M_{i, j}$ for all $j < \gamma$, $M_{i, 0} = M_i$, and $M_{i, \delta} = M_{i + 1}$). By strong universal local character at $\sigma$, for all $i < \alpha$, there exists $j_i < \gamma$ such that $p \rest M_{i + 1}$ does not $\ast$-fork over $M_{i, j_i}$. By replacing $j_i$ by $j_i + \sigma$ if necessary we can assume without loss of generality that $\cf{j_i} = \cf{\sigma}$.

    Observe also that for any $i < \alpha$, $p \rest M_{i + 1, j_i}$ $\ast$-forks over $M_i$ (using $(M)$ and the assumption that $p \rest M_{i + 1}$ $\ast$-forks over $M_i$). Therefore $\seq{M_0, M_{1, j_1}, M_2, M_{3, j_3}, \ldots}$, $p$ witness that $\adnf$ has $\sigma$-limit alternations at $\alpha$.

\end{enumerate}
\end{proof}

\begin{prop}\label{forking-calc-lem-part3}
Let $\alpha < \mu^+$ be a regular cardinal and $\sigma < \mu^+$ be a (not necessarily regular) cardinal
.  Assume that $\adnf$ has \emph{weak} universal local character at $\sigma$. If $\adnf$ has universal continuity at $\alpha$ and $\sigma$, $\adnf$ has no $\sigma$-limit alternations at $\omega$, and $\adnf$ has no $\sigma$-limit alternations at $\alpha$, then $\adnf$ has strong universal local character at $\alpha$.
\end{prop}

\begin{proof}
By Lemma \ref{forking-calc-lem-part1}.(\ref{part13}), $\adnf$ has strong universal local character at $\sigma$. By the contrapositive of Lemma \ref{forking-calc-lem-part1}.(\ref{part2}), $\adnf$ has weak universal local character at $\alpha$. By Lemma \ref{forking-calc-lem-part1}.(\ref{part1}), $\adnf$ has strong universal local character at $\alpha$.
\end{proof}

The next lemma corresponds to the second step outlined at the beginning of this section. Note that the added assumption is (\ref{hyp-9}) from the hypotheses of Theorem \ref{shvi-technical} and recall we are assuming Hypothesis \ref{hyp} throughout.

\begin{lemma}\label{technical-solv-lem}
  Assume $\K$ has an EM blueprint $\Phi$ with $|\tau (\Phi)| \le \mu$ such that every $M \in \K_{[\mu, \mu^+]}$ embeds inside $\EM_{\tau} (\mu^+, \Phi)$ . Let $\alpha < \mu^+$ be a regular cardinal. Then:

  \begin{enumerate}
  \item $\adnf$ has universal continuity at $\alpha$.
  \item If in addition $\alpha < \mu$, then for any limit $\gamma < \mu^+$, $\adnf$ has no $\gamma$-limit alternations at $\alpha$.
  \end{enumerate}
\end{lemma}
\begin{proof}
Let $\seq{M_i \mid i \leq \alpha}$ and $p$ be as in the definition of universal continuity or $\gamma$-limit alternations.  Let $S^{\mu^+}_\alpha := \{ \delta < \mu^+ \mid \cf \delta = \alpha\}$.  We say that $\bar C = \seq{C_\delta \mid \delta \in S^{\mu^+}_\alpha}$ is an \emph{$S^{\mu^+}_\alpha$-club sequence} if each $C_\delta \subseteq \delta$ is club.  Clearly, club sequences exist: just take $C_\delta := \delta$ (this will be enough for proving universal continuity). Shelah \cite{shg} proves the existence of club-guessing club sequences in ZFC under various hypotheses (the specific result that we use will be stated later, see Fact \ref{club-guess-fact-1}).  We will describe a construction of a sequence of models $\bar N(\bar C)$ based on a club sequence and then plug in the necessary club sequence in each case.

Given an $S^{\mu^+}_\alpha$-club sequence $\bar C$, enumerate $C_\delta \cup \{\delta\}$ in increasing order as $\seq{\beta_{\delta, j} \mid j \le \alpha}$.

{\bf Claim:}
  Let $\gamma < \mu^+$ be a limit ordinal. We can build increasing, continuous $\bar N(\bar C) = \seq{N_i \in \K^\ast \mid i < \mu^+}$ such that  for all $i < \mu^+$:
\begin{enumerate}
	\item $N_{i+1}$ is $(\mu, \gamma)$-limit over $N_i$;
	\item when $i \in S^{\mu^+}_\alpha$, there is $g_i:M_\alpha \cong N_i$ such that $g_i(M_j) = N_{\beta_{i, j}}$ for all $j \leq \alpha$; and:
	\item when $i \in S^{\mu^+}_\alpha$, there is $a_i \in N_{i+1}$ that realizes $g_i(p)$.
\end{enumerate}
{\bf Proof of Claim:} Build the increasing continuous chain of models as follows: start with an amalgamation base $N_0$, which exists by Hypothesis \ref{hyp}.(\ref{hyp-hyp-3}).  Given an amalgamation base $N_i$, build $N_{i+1}$ to be $(\mu, \gamma)$-limit over it. This exists by Hypothesis \ref{hyp}.(\ref{hyp-hyp-4}) of Theorem \ref{shvi-technical}), and $N_{i+1}$ is an amalgamation base by Hypothesis \ref{hyp}.(\ref{hyp-hyp-5}).  At limits, it also guarantees we have an amalgamation base.

At limits $i$ of cofinality $\alpha$, use the uniqueness of $(\mu, \gamma)$-limits models to find the desired isomorphisms: the weak version gives $M_0 \cong M_{\beta_{i, 0}}$, and the strong (over the base) version allows this isomorphism to be extended to get an isomorphism $g_i$ between $\seq{M_j \mid j \leq \alpha}$ and $\seq{N_{\beta_{i, j}} \mid j \leq \alpha}$ as described.  Since $N_{i+1}$ is universal over $N_i$, we there is some $a_i \in N_{i+1}$ that realizes $g_i(p)$. \hfill $\dag_{\text{Claim}}$\\

By assumption, we may assume that $N:= \bigcup_{i < \mu^+} N_i \leas \EM_\tau(\mu^+, \Phi)$.  Thus, we can write $a_i = \rho_i(\gamma_1^i, \dots, \gamma_{n(i)}^i)$ with:

$$
\gamma^i_1< \dots <\gamma^i_{m(i)} < i \leq \gamma^i_{m(i)+1} < \dots< \gamma^i_{n(i)} < \mu^+
$$

Now we begin to prove each part of the lemma.  In each, we will find $i_1 < i_2 \in S^{\mu^+}_\alpha$ such that $\gtp(a_{i_1}/N_{i_1}; N)$ and $\gtp(a_{i_2}/N_{i_1}; N)$ are both the same (because of the EM structure) and different (because they exhibit different $*$-forking behavior), which is our contradiction.

\begin{enumerate}
  \item Assume that $p \rest M_{j}$ does not fork over $M_0$, for all $j < \alpha$.

    Let $\bar{C}$ be an $S^{\mu^+}_\alpha$-club sequence, and set $\seq{N_i \in \K^\ast \mid i < \mu^+} = \bar N(\bar C)$ as in the Claim (the value of $\gamma$ doesn't matter here, e.g.\ take $\gamma := \omega$).  By Fodor's Lemma, there is a stationary subset $S^* \subseteq S^{\mu^+}_\alpha$, a term $\rho_\ast$, $m_\ast, n_\ast < \omega$ and ordinals $\gamma_{0}^\ast, \ldots \gamma_{n_\ast}$, $\beta_{\ast, 0}$ such that: \bigskip
    
    \begin{enumerate}
    \item[] For every $i \in S^*$, we have $\rho_i = \rho_*$; $n(i) = n_*$; $m(i) = m_*$; $\gamma^i_j = \gamma^*_j$ for $j \leq m_*$; and $\beta_{i,0} = \beta_{*,0}$.
    \end{enumerate}

    \bigskip
    
    Set $E := \{\delta < \mu^+ \mid \delta$ is limit and $\EM_{\tau} (\delta, \Phi) \cap N = N_\delta\}$.  This is a club.  Let $i_1 < i_2$ both be in $S^* \cap E$.  Then we have:

    \begin{align*}
      \gtp\left(a_{i_1}/N_{i_1}\right) &= \gtp\left(\rho_*(\gamma^*_1, \dots,\gamma^*_{m_*}, \gamma^{i_1}_{m_*+1}, \dots, \gamma^{i_1}_{n_*})/N \cap \EM_\tau(i_1, \Phi)\right)\\
      &=\gtp\left(\rho_*(\gamma^*_1, \dots, \gamma^*_{m_*}, \gamma^{i_2}_{m_*+1}, \dots, \gamma^{i_2}_{n_*})/N \cap \EM_\tau(i_1, \Phi)\right) \\ 
      &= \gtp\left(a_{i_2}/N_{i_1}\right)
    \end{align*}

    where all the types are computed inside $N$. This is because the only differences between $a_{i_1}$ and $a_{i_2}$  lie entirely above $i_1$.
    
    We have that $g_{i_1}:(N_{i_1}, N_{\beta_{*, 0}}) \cong (M_\alpha, M_0)$ and that $p$ $*$-forks over $M_0$.  Thus, $\gtp(a_{i_1}/N_{i_1}) = g_{i_1}(p)$ $*$-forks over $N_{\beta_{*, 0}}$.  On the other hand, $C_{i_2}$ is cofinal in $i_2$, so there is $j < \alpha$ such that $\beta_{i_2, j} > i_1$ and, thus, $N_{i_1} \leas N_{\beta_{i_2, j}}$.  Again, $g_{i_2}:(N_{\beta_{i_2, j}}, N_{\beta_{*, 0}}) \cong (M_j, M_0)$ and $p \rest M_j$ does not $*$-fork over $M_0$ by assumption.  Thus, $\gtp(a_{i_2}/N_{\beta_{i_2,j}}) = g_{i_2} (p \rest M_j)$ does not $*$-fork over $N_{\beta_{*, 0}}$.  By monotonicity (M), $\gtp(a_{i_2}/N_{i_1})$ does not $*$-fork over $N_{\beta_{*,0}}$.  Thus, $\gtp(a_{i_1}/N_{i_1}) \neq \gtp(a_{i_2}/N_{i_2})$, a contradiction.

  \item  Let $\chi$ be a big-enough cardinal and create an increasing, continuous elementary chain of models of set theory $\seq{\mathfrak{B}_i \mid i < \mu^+}$ such that for all $i < \mu^+$:
    
\begin{enumerate}
	\item $\mathfrak{B}_i \prec (H(\chi), \in)$;
	\item $\|\mathfrak{B}_i\| = \mu$;
	\item \label{things} $\mathfrak{B}_0$ contains, as elements\footnote{When we say that $\mathfrak{B}_0$ contains a sequence as an element, we mean that it contains the function that maps an index to its sequence element.}, $\Phi$, $\EM(\mu^+,\Phi)$, $h$, $\mu^+$, $\seq{N_i \mid i < \mu^+}$, $S^{\mu^+}_\alpha$, $\seq{a_i \mid i \in S^{\mu^+}_\alpha}$,  and each $f \in \tau(\Phi)$; and 
	\item $\mathfrak{B}_i \cap \mu^+$ is an ordinal.
\end{enumerate}

We will use the following fact which was originally proven in \cite[III.2]{shg} (or see \cite[Theorem 2.17]{cardarithm} for a short proof).

\begin{fact} \label{club-guess-fact-1}
Let $\lambda$ be a cardinal such that $\cf \lambda \geq \theta^{++}$ for some regular $\theta$ and let $S \subseteq S^\lambda_\theta$ be stationary.  Then there is a $S$-club sequence $\seq{C_\delta \mid \delta \in S}$ such that, if $E \subseteq \lambda$ is club, then there are stationarily many $\delta \in S$ such that $C_\delta \subseteq E$.
\end{fact}


We have that $\alpha < \mu$, so we can apply Fact \ref{club-guess-fact-1} with $\lambda, \theta, S$ there standing for  $\mu^+, \alpha, S^{\mu^+}_{\alpha}$ here. Let $\bar{C}$ be the $S^{\mu^+}_\alpha$-club sequence that the fact gives. Let $\seq{N_i \in K_\mu \mid i < \mu^+} = \bar N(\bar C)$ be as in the Claim.  Note that $E := \{ i < \mu^+ \mid \mathfrak{B}_i \cap \mu^+ = i\}$ is a club.  By the conclusion of Fact \ref{club-guess-fact-1}, there is some $i_2 \in S^{\mu^+}_{\alpha}$ such that $C_{i_2} \subseteq E$.  We have $a_{i_2} = \rho_{i_2}(\gamma_1^{i_2}, \dots, \gamma_{n(i_2)}^{i_2})$, with:

$$
\gamma_1^{i_2} < \dots < \gamma_{m(i_2)}^{i_2} < i_2 \leq \gamma^{i_2}_{m(i_2)+1} < \dots < \gamma^{i_2}_{n(i_2)}
$$

Since the $\beta_{i_2, j}$'s enumerate a cofinal sequence in $i_2$, we can find $j < \alpha$ such that $\gamma_{m(i_2)}^{i_2} < \beta_{i_2, 2j+1} < i$.  Recall that we have $p \rest M_{2j+2}$ does not $\ast$-fork over $M_{2j+1}$ by assumption.  Then $(H(\chi), \in)$ satisfies the following formulas with parameters exactly the objects listed in item (\ref{things}) above and ordinals below $\beta_{i_2, 2j+2}$:

\begin{align*}
  \exists x, y_{m(i_2)+1}, \ldots, y_{n(i)}.& ( ``x \in S^{\mu^+}_{\alpha}" \\
  &\wedge ``x > \beta_{i_2, 2j+1}" \wedge ``y_k \in (x, \mu^+) \text{ are increasing ordinals}"\\
  &\wedge ``a_x = \rho_{i_2}(\gamma_1^{i_2}, \dots, \gamma_{m(i_2)}^{i_2}, y_{m(i_2)+1}, \dots, y_{n(i_2)})" \\
  &\wedge ``N_x \subset \EM(x, \Phi)" ) \\
\end{align*}

This is witnessed by $x = i_2$ and $y_k = \gamma^{i_2}_k$.  By elementarity, $\mathfrak{B}_{\beta_{_2, 2j+2}}$ satisfies this formula as it contains all the parameters.  Let $i_1 \in (\beta_{i_2, 2j+1}, \mu^+) \cap \mathfrak{B}_{\beta_{i_2, 2j+2}} = (\beta_{i_2, 2j+1}, \beta_{i_2, 2j+2})$\footnote{The equality here is the key use of club guessing.} witness this, along with $\gamma'_{m(i_2)+1} < \dots < \gamma'_{n(i_2)} < \mu^+$.  Then we have:

$$a_{i_1} = \rho_{i_2}(\gamma_1^{i_2}, \dots, \gamma_{m(i_2)}^{i_2}, \gamma'_{m(i_2)+1}, \dots, \gamma'_{n(i_2)})$$

with $\beta_{i_2, 2j+1} < \gamma_{m(i_2)+1}$.  We want to compare $\gtp(a_{i_2}/N_{i_1})$ and $\gtp(a_{i'}/N_{i_1})$.

\begin{itemize}

	\item From the elementarity, we get that $N_{i_1} \subseteq \EM_\tau(i_1, \Phi)$.  We also know that $i_1 < \beta_{i_2, 2j+2} < \gamma^{i_2}_{m(i_2)+1}, \gamma'_{m(i_2)+1}$.  Thus, as before, the types are equal.

	\item We know that $p \rest M_{2j+2}$ does not $*$-fork over $M_{2j+1}$.  Thus, $\gtp(a_{i_2}/N_{\beta_{i_2, 2j+2}})$ does not $*$-fork over $N_{\beta_{i_2, 2j+1}}$.  Since we have $N_{\beta_{i_2, 2j+1}} \leas N_{i_1} \ltas N_{\beta_{i_2, 2j+2}}$, this gives $\gtp(a_{i_2}/N_{i_1})$ does not $*$-fork over $N_{\beta_{i_2, 2j+1}}$.
	
	\item We have $\beta_{i_2, 2j+1} < i_1$, so there is some $k < \alpha$ such that $\beta_{i_2, 2j+1} < \beta_{i_1, k} < i'$.  By assumption, $p$ $*$-forks over $M_k$.  Thus $g_{i_1}(p)$ $*$-forks over $N_{\beta_{i_1,k}}$. Therefore $\gtp(a_{i_1}/N_{i_1})$ $*$-forks over $N_{\beta_{i_2, 2j+1}} \leas N_{\beta_{i_1,k}}$.

\end{itemize}

As before, these three statements contradict each other.
\end{enumerate}
\end{proof}

We now prove the main theorem, Theorem \ref{shvi-technical}.  Recall that the assumptions of this theorem include the main context of this section (Hypothesis \ref{hyp}); $\adnf$ has weak universal local character somewhere; and $\K$ has an $EM$ blueprint that every model embeds into.

\begin{proof}[Proof of Theorem \ref{shvi-technical}]
Pick a cardinal $\sigma < \mu^+$ such that $\adnf$ has weak universal local character at $\sigma$ (exists by assumption (\ref{hyp-8})).

  As announced at the beginning of this section, our proof of Theorem \ref{shvi-technical} really has two steps: a forking calculus step (implemented in Lemmas \ref{forking-calc-lem-part0} and \ref{forking-calc-lem-part1} and Proposition \ref{forking-calc-lem-part3}) and a set-theoretic step (implemented in Lemma \ref{technical-solv-lem}). The claim below is key. The work done in the first step will show that the claim suffices, and the second step will prove the claim.

  \underline{Claim}: For any limit ordinal $\gamma < \mu^+$ and any regular cardinal $\alpha < \mu^+$, $\adnf$ has universal continuity at $\alpha$ and no $\gamma$-limit alternations at $\alpha$.

  By Proposition \ref{forking-calc-lem-part3}, the claim implies that $\adnf$ has strong universal local character at any regular $\alpha < \mu^+$. This suffices by Remark \ref{loc-rmk}. It remains to prove the claim.

  \underline{Proof of Claim}: Universal continuity holds by Lemma \ref{technical-solv-lem}. When $\alpha < \sigma$, Lemma \ref{technical-solv-lem} also gives that $\adnf$ has no $\gamma$-limit alternations at $\alpha$. Assume now that $\alpha \ge \sigma$. By Remark \ref{loc-rmk}, $\adnf$ has weak universal local character at any limit $\sigma' \in [\sigma, \mu^+)$, so in particular in $\alpha$. By Lemma \ref{forking-calc-lem-part0}, $\adnf$ has no $\gamma$-limit alternations at $\alpha$, as desired. $\dagger_{\text{Claim}}$
\end{proof}

\bibliographystyle{amsalpha}
\bibliography{SVTheoremNotes}

\end{document}